\documentclass[conference]{IEEEtran}

\IEEEoverridecommandlockouts
\usepackage{amsthm}
\usepackage{booktabs}
\usepackage[table,dvipsnames]{xcolor}
\newtheorem{theorem}{Theorem}
\newtheorem{definition}{Definition}

\newtheorem{corollary}{Corollary}
\newtheorem{proposition}{Proposition}
\newtheorem{remark}{Remark}
\usepackage[ruled]{algorithm2e}
\usepackage{bbm}
\usepackage{cite}
\usepackage{amsmath,amssymb,amsfonts,dsfont}
\usepackage{algorithmic}
\usepackage{graphicx}
\usepackage{textcomp}
\usepackage{xcolor}
\usepackage{hyperref}
\def\BibTeX{{\rm B\kern-.05em{\sc i\kern-.025em b}\kern-.08em
    T\kern-.1667em\lower.7ex\hbox{E}\kern-.125emX}}
\begin{document}

\title{Scalable Min-Max Optimization via Primal-Dual Exact Pareto Optimization\thanks{This work was supported by
EPSRC Grants EP/X04047X/1 and EP/Y037243/1.}}

\author{\IEEEauthorblockN{Sangwoo Park}
\IEEEauthorblockA{\textit{Imperial College London}\\
London, United Kingdom \\
s.park@imperial.ac.uk}
\and
\IEEEauthorblockN{Stefan Vlaski}
\IEEEauthorblockA{\textit{Imperial College London}\\
London, United Kingdom \\
s.vlaski@imperial.ac.uk}
\and
\IEEEauthorblockN{Lajos Hanzo}
\IEEEauthorblockA{\textit{University of Southampton}\\
Southampton, United Kingdom \\
hanzo@soton.ac.uk}
}

\definecolor{LH}{RGB}{0,0,0}
\maketitle

\begin{abstract}
In multi-objective optimization, minimizing the worst objective can be preferable to minimizing the average objective, as this ensures {\color{LH}improved} fairness across objectives. Due to the non-smooth nature of the result{\color{LH}ant} min-max optimization problem, classical subgradient{\color{LH}-}based approaches {\color{LH}typically} exhibit slow convergence. Motivated by primal-dual consensus techniques in multi-agent optimization and learning, we formulate a smooth variant of the min-max problem based on the augmented Lagrangian. The result{\color{LH}ant} Exact Pareto Optimization via Augmented Lagrangian (EPO-AL) algorithm scales better with the number of objectives than subgradient{\color{LH}-}based strategies, while exhibiting lower per-iteration complexity than recent smoothing{\color{LH}-}based {\color{LH}counterparts}. We establish that {\color{black} every fixed-point} of the proposed algorithm is both Pareto and min-max optimal under mild assumptions and demonstrate its effectiveness in numerical simulations.
\end{abstract}


\section{Introduction} \label{sec:intro}
We consider a multi-objective optimization problem {\color{LH}having} $K$ differentiable, positive objectives $J_1(w), \ldots, J_K(w)>0$ where $J_k(w)$ is the $k$-th objective evaluated at the model $w \in \mathbb{R}^d$. We wish to solve the \emph{weighted min-max} problem \cite{miettinen1999nonlinear, fei2016survey}:
\begin{align} \label{eq:obj}
    \min_{w \in \mathbb{R}^d} \max_{k \in [K]} r_k J_k(w),
\end{align}
given a pre-determined \emph{preference vector} $r = [r_1,\ldots,r_K]^\top$ {\color{LH}associated} with $r_k > 0$ for $k=1,\ldots,K$. We focus {\color{LH}our attention} on the setting where the gradient information $\{\nabla J_k(w)\}_{k=1}^K$ is available. Solving the min-max problem~\eqref{eq:obj} can be preferable to classical linear scalarization, i.e., $\min_{w \in \mathbb{R}^d}\sum_{k=1}^K r_k J_k(w)$, in applications where fairness is important \cite{simeone2018very, mohri2019agnostic, hou2022joint, wang2024hybrid, hamidi2025over}{\color{LH},}  since it ensures that no individual objective \( J_k(\cdot) \) is neglected in the interest of {\color{LH}improving the} average performance.


\begin{figure}[t!]
\begin{center}\includegraphics[scale=0.09]{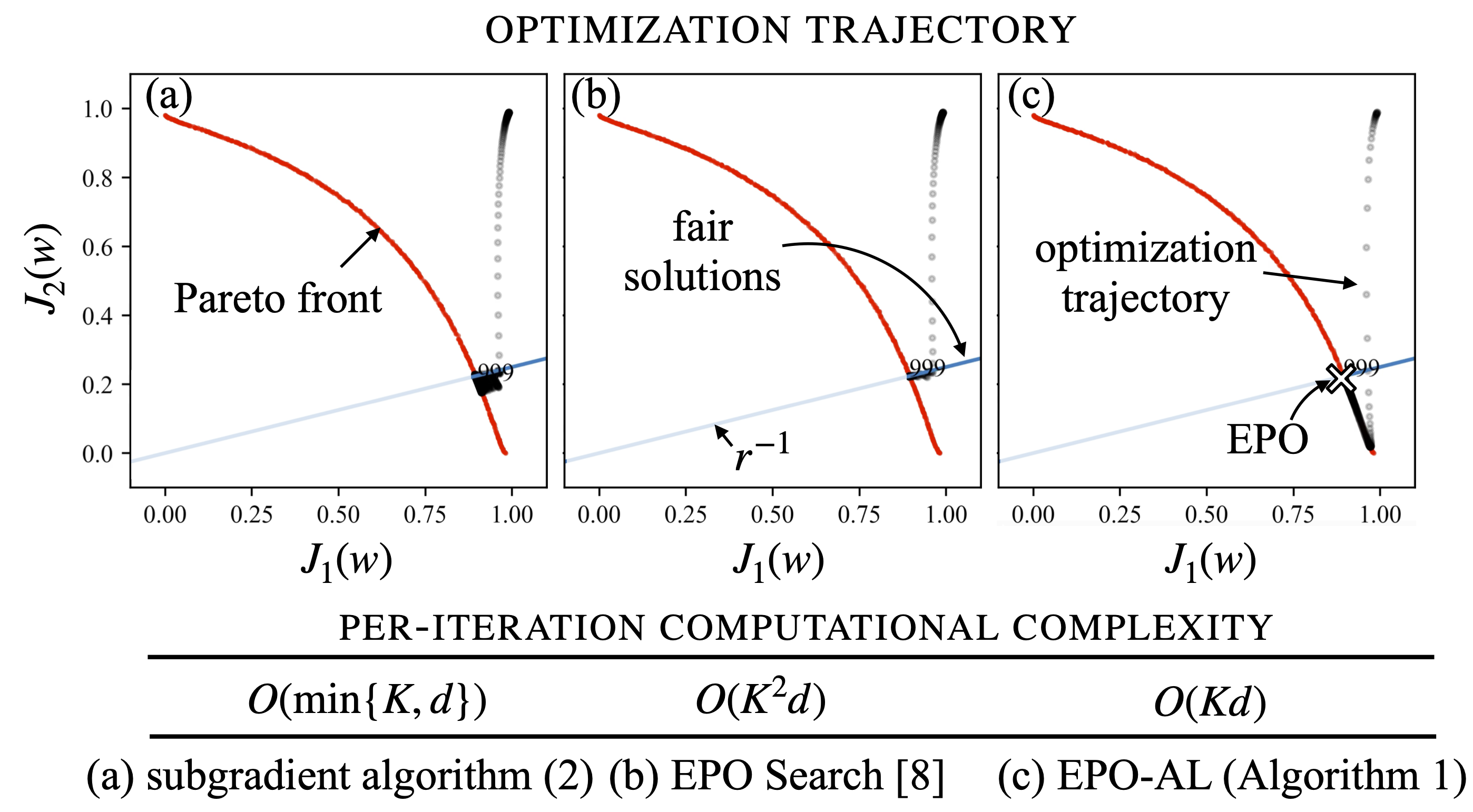}
    \end{center}
    \caption{Multi-objective optimization trajectories (top) for \textcolor{black}{subgradient} algorithm (\ref{eq:subgrad}) \textcolor{black}{in (a)}, EPO Search \cite{mahapatra2020multi} \textcolor{black}{in (b)}, and the proposed approach via augmented Lagrangian \textcolor{black}{in (c)}, referred to as EPO-AL; the table shows the per-iteration computational complexities (bottom). Optimization trajectories are obtained for $K=2$ non-convex objectives $J_1(w) = 1 - e^{-\|w-1/\sqrt{d}\|^2}$ and $J_2(w) = 1 - e^{-\|w+1/\sqrt{d}\|^2}$~\cite{lin2019pareto} with \( w \in \mathds{R}^3 \) (see Sec.~\ref{subsec:vis} for details). {\color{black} The intersection between the Pareto front (red arc, see Def.~\ref{def:weak_PO}) and the fair solutions (blue line, see Def.~\ref{def:FO}) is an exact Pareto optimal (EPO) \cite{mahapatra2020multi} solution (white cross, see Def.~\ref{def:epo}), which satisfies the min-max optimality (\ref{eq:obj}) under mild assumptions (see Prop.~\ref{prop:epo_minmax}).} Observe that the proposed strategy first finds the Pareto front, and then searches for the Pareto solution that is min-max optimal according to~\eqref{eq:obj}. {\color{black} The subgradient algorithm (\ref{eq:subgrad}) exhibits oscillations around the min-max optimal solution Pareto front due to the non-smooth behavior of maximum operator in (\ref{eq:obj}), unlike both EPO-based approaches that smoothly converge to the min-max optimal solution.}}
    \label{fig:vis}
    \vspace{-0.3cm}
 \end{figure}

Perhaps the {\color{LH}conceptually} simplest approach for solving (\ref{eq:obj}) is to consider the \textcolor{black}{subgradient algorithm} \cite[Theorem 18.5]{bauschke2017correction}:
\begin{align} \label{eq:subgrad}
    w_{i+1} = w_i - \mu \cdot r_{k^\text{active}} \nabla J_{k^\text{active}}(w_i)
\end{align}
where \( \mu \) is a step-size and $k^\text{active} \triangleq \arg\max_k r_k J_k(w)$ denotes the index for the \emph{active} objective that attains the maximum in~\eqref{eq:obj}, i.e., $\max_{k} r_k J_k(w) = r_{k^\text{active}}J_{k^\text{active}}(w)$. Note that, if more than one objective achieves the maximum for a particular model $w$, i.e., the set $\mathcal{A}(w) = \{k'\in[K]: J_{k'}(w) = \max_{k} J_{k}(w) \}$ contains more than one element, then any convex combination of $\{\nabla J_k(w)\}_{k \in \mathcal{A}(w)}$ is a subgradient of the min-max objective~\eqref{eq:obj} and can be utilized in~\eqref{eq:subgrad}~\cite{ras2024identification}.

However, the iterative update rule (\ref{eq:subgrad}) generally suffers from slow convergence rate \cite{ras2024identification}, primarily due to {\color{LH}ignoring the} gradient information {\color{LH}gleaned} from inactive objectives. This observation has motivated the development of smoothing{\color{LH}-}based alternatives~\cite{zang1980smoothing, gokcesu2019accelerating, mohri2019agnostic, ras2024identification}. Existing approaches for~\eqref{eq:obj} either rely on (\emph{i}) directly replacing the max-function by a smooth approximation \cite{zang1980smoothing, epasto2020optimal, hou2022joint, lin2024smooth, gokcesu2019accelerating},  or (\emph{ii}) designing {\color{LH}a} smooth saddle point problem \cite{mohri2019agnostic, ras2024identification, hamidi2025over}, i.e.,
\begin{align} \label{eq:saddle}
    w^\star \in \arg \min_{w\in\mathbb{R}^d} \max_{y \in \Delta^{K}} \sum_{k=1}^K r_kJ_k(w) y_k,
\end{align}
where $y_k$ is the $k$-th element of the $(K-1)$-simplex $y \in \Delta^{K}$ given by $\Delta^{K} = \{ y \in \mathbb{R}^{K}_+: y^\top \mathds{1}_{K} = 1\}$ with  $\mathds{1}_K$ being the $K\times 1$ all-one vector.


\section{Preliminaries} \label{sec:prelim}
In this paper, we consider an alternative smoothing approach for min-max optimization problem (\ref{eq:obj}) inspired by classical results in multi-objective optimization capturing the relationship between min-max and Pareto optimization~\cite{lin2005min}. To formalize the discussion, we introduce the following definitions.
\begin{definition}[Weak Pareto optimality \cite{miettinen1999nonlinear}] \label{def:weak_PO}
    A model $w$ is weakly Pareto optimal{\color{LH},} if there exists no $w' \neq w$ for which all objectives are reduced. Accordingly, the collection $\mathcal{P}$ of all weakly Pareto optimal points  is given by
    \begin{align} \label{eq:weak_PO}
        \mathcal{P} = \big\{ w \in \mathbb{R}^d: \nexists w'\in \mathbb{R}^d \text{ s.t. } {J}_k(w') < J_k(w) \text{ for all } k\big\}.
    \end{align}    
\end{definition}
The set of (weakly) Pareto optimal points is illustrated as a red arc in Fig.~\ref{fig:vis}. Second, we introduce the set of {\color{black}\emph{fair}} points~\cite{li2019fair, hamidi2025over}.

\begin{definition}[{\color{black}Fairness}~\cite{li2019fair, hamidi2025over}]\label{def:FO} A model $w$ is called {\color{black} fair} with respect to the preference vector $r$ if the weighted objectives are equal, i.e., $r_1 J_1(w)=\cdots=r_K J_K(w)$. The collection of all such {\color{black}fair} models is denoted by the set:
\begin{align} \label{eq:pf_pefect_fair}
     \mathcal{F}_r = \big\{ w \in \mathbb{R}^d: r_1 J_1(w) = \cdots = r_K J_K(w) \}.
\end{align}    
\end{definition} 
The set of {\color{black} fair} points is illustrated as a blue line in Fig.~\ref{fig:vis}. When the set of (weakly) Pareto optimal solutions and the set of {\color{black}fair solutions} intersect, this gives rise to the set of \emph{exact Pareto optimal} solutions {\color{black}(white cross in Fig.~\ref{fig:vis})}.
\begin{definition}[Exact Pareto optimality~\cite{mahapatra2020multi}] \label{def:epo}
    The set of exact Pareto optimal solutions is given by:
    \begin{align}
        \mathcal{E}_r = \mathcal{P} \cap \mathcal{F}_r.
    \end{align}
\end{definition}
The set \( \mathcal{E}_r \) is given by the intersection of the red arc (the weakly Pareto optimal solutions) and the blue line (the {\color{black}fair solutions}) in Fig.~\ref{fig:vis}. When this intersection is non-empty, we refer to \( r \) as \emph{{\color{LH}being} Pareto feasible}. We remark that compared to~\cite[Eq. (8)]{mahapatra2020multi} we define \( \mathcal{E}_r \) in terms of global (weak) Pareto optimality, rather than merely local Pareto optimality. This allows us to to establish the following proposition.
\begin{proposition}[Exact Pareto optimality implies min-max optimality]\label{prop:epo_minmax} 
    {\color{LH}Assume that} \( w^{\mathrm{EPO}} \in \mathcal{E}_r = \mathcal{P} \cap \mathcal{F}_r \) is {\color{black}weakly Pareto optimal and fair}. Then \( w^{\text{WPO}} \) is also min-max optimal as defined in~\eqref{eq:obj}:
    \begin{align}
        w^{\mathrm{EPO}} = \arg\min_{w \in \mathbb{R}^d} \max_{k \in [K]} r_k J_k(w).
    \end{align}
    \end{proposition}
    \begin{proof}
    We prove the statement by contradiction. Suppose that $w^\text{EPO} \in \mathcal{P}\cap \mathcal{F}_r$ does not does not minimize (\ref{eq:obj}). Then: 
    \begin{align}
        \exists w: \max_{k\in[K]} r_k J_k(w) < \max_{k\in[K]} r_k J_k(w^\text{EPO}).
    \end{align}
    We first simplify the right-hand side by using the {\color{black}fairness condition }(\ref{eq:pf_pefect_fair}). From $w^\text{EPO}\in\mathcal{F}_r$, we have
    \begin{align}
        \max_{k\in[K]} r_k J_k(w) <  r_{\ell} J_{\ell}(w^\text{EPO}) \text{ for any ${\ell}\in [K]$}.
    \end{align}
    Since each element of a set is upper bounded by its maximum:
    \begin{align} \label{eq:pf_last}
        r_{\ell} J_{\ell}(w) \le \max_{k\in[K]} r_k J_k(w) <  r_{\ell} J_{\ell}(w^\text{EPO}) \text{ for all $\ell \in [K]$}.
    \end{align}
    After cancelling~\( r_{\ell} > 0 \) on both sides of~\eqref{eq:pf_last}, we conclude that \( w^{\mathrm{EPO}} \) cannot be weakly Pareto optimal. Hence \( w^\text{EPO} \notin \mathcal{P}\cap \mathcal{F}_r \), leading to a contradiction. \end{proof}

Proposition~\ref{prop:epo_minmax} provides a sufficient condition to ensure that a solution for the min-max problem~\eqref{eq:obj} can be pursued by instead searching for a point on the (weak) Pareto frontier \( \mathcal{P} \) that is also {\color{black}fair} \( \mathcal{F}_r \). This fact does not hold in general --- see~\cite{lin2005min} for counter{\color{LH}-}examples. As long as \( \mathcal{E}_r = \mathcal{P} \cap \mathcal{F}_r \) is non-empty, however, it follows that any exact Pareto optimal point is also min-max optimal. Motivated by this consideration, in the sequel we will propose a new algorithm for min-max optimization via exact Pareto optimization. Compared to existing algorithms in the literature, the proposed strategy will rely on a single-time scale, and exhibit a reduced per-iteration complexity of \( O(K d) \) as opposed to \( O(K^2d)\) \cite{mahapatra2020multi, mahapatra2021exact, momma2022multi}. This results in better scaling with the number of objectives \( K \).

\section{Exact Pareto Optimality via Augmented Lagrangian} \label{sec:deriv_algo}
In this section, we develop an algorithm for exact Pareto optimization via the augmented Lagrangian, inspired by classical primal-dual consensus techniques in multi-agent optimization and learning --- see~\cite{Schizas08,Towfic15} for early examples and~\cite{Vlaski23} for a recent survey. To this end, note that an exact Pareto optimal solution can be pursued via:
\begin{subequations}
    \begin{align} 
        &\min_{w \in \mathbb{R}^d}  \frac{1}{K}\sum_{k=1}^K J_k(w) \label{eq:constraint_main_pre}\\
        &\text{s.t. } r_k J_k(w) = r_{\ell}J_{\ell}(w) \ \ \ \forall k, \ell. \label{eq:constraint_const_pre}
    \end{align}
\end{subequations}
Here, {\color{LH}the} relation{\color{LH}ship}~\eqref{eq:constraint_main_pre} encourages Pareto optimality, while~\eqref{eq:constraint_const_pre} ensures the fairness {\color{black}condition}~\eqref{eq:pf_pefect_fair}. In light of~\eqref{eq:constraint_const_pre}, the objective~\eqref{eq:constraint_main_pre} can be replaced by any linear combination of \( J_{k}(\cdot) \) {\color{LH}having} non-negative weights. The choice of equal weighting given by \( \frac{1}{K} \) is merely a matter of simplicity. In analogy to~\cite{Schizas08,Towfic15,Vlaski23}, we replace the collection of constraints~\eqref{eq:constraint_const_pre} by a single constraint involving the aggregate constraint violations:
\begin{subequations}\label{eq:constraint}
\begin{align} 
    &\min_{w \in \mathbb{R}^d}  \frac{1}{K}\sum_{k=1}^K J_k(w) \label{eq:constraint_main}\\
    &\text{s.t. } \frac{1}{2 K}\sum_{k=1}^K \sum_{\ell=1}^K  \| r_k J_k(w) - r_{\ell}J_{\ell}(w)\|^2=0. \label{eq:constraint_const}
\end{align}
\end{subequations}
The fairness {\color{black}condition}~\eqref{eq:constraint_const} can be written more compactly as:
\begin{align}\label{eq:helloooo}
    &\frac{1}{K}\sum_{k=1}^K \sum_{\ell=1}^K  \| r_k J_k(w) - r_{\ell}J_{\ell}(w)\|^2 = \mathcal{J}(w)^{\top} {L_r}\mathcal{J}(w)  =0{\color{LH},}
\end{align}
where $\mathcal{J}(w) = [J_1(w),\ldots,J_K(w)]^\top$ is a vector containing the collection of objectives evaluated {\color{LH}for the model} $w$  and $L_r$ is given by:
\begin{align} \label{eq:L_r}
    L_r = \text{diag}(r)\Big(I_{K\times K} -\frac{1}{K} \mathds{1}_K \mathds{1}_K^\top\Big)\text{diag}(r).
\end{align}
\textcolor{black}{Here,} $\text{diag}(r)$ is a diagonal matrix that contains the elements of $r$ on its diagonal. Since \( L_r \) is symmetric and positive semi-definite, it has a square root \( \sqrt{L_r} \) that satisfies \( \sqrt{L_r} \sqrt{L_r} = L_r \) and hence~\eqref{eq:constraint_const} is equivalent to:
\begin{align}
    \left\| \sqrt{L_r} \mathcal{J}(w) \right\|^2 =0 \Longleftrightarrow  \sqrt{L_r}\mathcal{J}(w) = 0.
\end{align}

We can then define the corresponding augmented Lagrangian~\cite[Sec. 4]{bertsekas1997nonlinear} as
\begin{align}
    \mathcal{L}(w,\lambda) &=\frac{1}{K}\mathds{1}_K^\top \mathcal{J}(w)  + \lambda^\top \sqrt{L_r}\mathcal{J}(w) + \frac{\eta}{2}\big|\big| \sqrt{L_r}\mathcal{J}(w)\big|\big|^2, 
\end{align}
where $\eta > 0$ is a penalty parameter and $\lambda$ is the corresponding Lagrangian multiplier. We then update the \emph{primal} variable $w$ and the \emph{dual} variable $\lambda$  in an iterative first-order approach as in \cite{arrow1958studies, alghunaim2020linear}:
\begin{subequations}
\begin{align}
    \label{eq:lag_pri}
    w_{i} &= w_{i-1} - \mu \nabla_{w} \mathcal{L}(w_{i-1},\lambda_{i-1}) \\&= w_{i-1} - \mu G(w_{i-1})\left[ \frac{1}{K}\mathds{1}_K + \sqrt{L_r}\lambda_{i-1} + \eta L_r \mathcal{J}(w_{i-1})\right] \nonumber\\
    \lambda_{i} &= \lambda_{i-1} + \mu \nabla_\lambda \mathcal{L}(w_{i-1}, \lambda_{i-1}) = \lambda_{i-1} + \mu \sqrt{L_r}\mathcal{J}(w_{i-1}),\label{eq:lag_dual}
\end{align}    
\end{subequations}
where $G(w) = \big[ \nabla J_1(w),\ldots, \nabla J_K(w)\big]$ is a $d \times K$ matrix that collects the gradients from the $K$ objectives evaluated {\color{LH}for the} model $w$ and $\mu>0$ is the step size.
Multiplying (\ref{eq:lag_dual}) by \( \sqrt{L_r} \) from the left and by defining  $p_i \triangleq (1/K)\mathds{1}_K + \sqrt{L_r}\lambda_i $, we obtain the {\color{LH}following} equivalent formulation:
\begin{subequations}
    \begin{align}
    w_{i} &= w_{i-1} - \mu G(w_{i-1})\big[p_{i-1} + \eta L_r\mathcal{J}(w_{i-1})\big] \label{eq:lag_last_primal}\\
    p_{i} &= p_{i-1} +\mu L_r \mathcal{J}(w_{i-1}).  \label{eq:lag_last_dual}
\end{align}
\end{subequations}
From the initialization $\lambda_0=0$, we find the initial condition $p_0 = (1/K)\mathds{1}_K$. Lastly, by we apply the \textcolor{black}{positivity} operator $[\cdot]_+ = \max\{ \cdot, 0\}$ element-wise to $p_{i-1}$ in (\ref{eq:lag_last_primal}), which yields Algorithm~\ref{alg:PO_minmax}.

\begin{algorithm}[h] 
    \DontPrintSemicolon
    \smallskip 
    \KwIn{ $K$ positive, differentiable, objectives $J_1(\cdot), ..., J_K(\cdot)$ with $J_k: \mathbb{R}^d \rightarrow \mathbb{R}_+$; step size $\mu>0$;  penalty parameter $\eta > 0$.}
    \vspace{0.15cm}
    \hrule
    \vspace{0.15cm}
    {\bf initialize} $w_0 \in \mathbb{R}^d, p_0 = \mathds{1}_K/K, z_t = \mathbf{0}_K$, where $\mathbf{0}_K$ is the all-zero vector of size $K$.
    
    \For{\em $i=1,2,\ldots$}{
    \vspace{-0.2cm}
    \begin{subequations}
    \begin{align}
        w_{i} &= w_{i-1} - \mu  G(w_{i-1}) \big([p_{i-1}]_+ + \eta  L_r \mathcal{J}(w_{i-1})\big) \label{eq:primal} \\
        p_{i} &= p_{i-1} +   \mu L_r \mathcal{J}(w_{i-1}) \label{eq:dual}
    \end{align}
\end{subequations}
    }
      
    \caption{Exact Pareto Optimization via Augmented Lagrangian (EPO-AL)}
    \label{alg:PO_minmax}
    \end{algorithm}

The per-iteration complexity of EPO-AL scales well with the number of objectives, as summarized in the following remark. 
\begin{remark}[Per-iteration computational complexity] \label{rmk:compute_complexity}
Each iteration of EPO-AL requires $O(K d)$ computations, resulting from the evaluation of the $d \times K$ gradient matrix $G(w_{i-1})$ and multiplication with the $K \times 1$ vector $([p_{i-1}]_+ + \eta L_r \mathcal{J}(w_{i-1}))$.
\end{remark}
Note that typical multi-objective optimization algorithms~\cite{chen2023three}, including the ones that aim {\color{LH}for} finding exact Pareto optimal solutions~\cite{ mahapatra2020multi, mahapatra2021exact, momma2022multi} generally involve the evaluation of $G(w_{i-1})^\top G(w_{i-1})${\color{LH},} which requires {\color{LH}on the order of} $O(K^2d)$ computations; subgradient-based approaches (\ref{eq:subgrad}) require $O(K)$ computations to identify the active objective, and \( O(d) \) computations to evaluate the corresponding subgradient.

\section{Fixed Point Analysis} \label{sec:fixed}
Before analyzing the fixed-point behavior of Algorithm~\ref{alg:PO_minmax}, we first recall the notion of \emph{Pareto stationarity} \cite{desideri2012multiple}. 

\begin{definition}[Pareto stationarity] \label{def:PS}
    A model is called Pareto stationary if one can find a convex combination of the gradients $\{\nabla J_k(w)\}_{k=1}^K$ that yields an all-zero vector. Hence, the collection of all Pareto stationary points $\mathcal{P}^\text{st}$ is given by:
    \begin{align}
        \mathcal{P}^\text{st} = \Big\{w \in \mathbb{R}^d: \min_{p \in \Delta^K}|| G(w) p || = 0 \Big\}.
    \end{align}
\end{definition}
Note that weak Pareto optimality implies Pareto stationarity, i.e., $\mathcal{P} \subseteq \mathcal{P}^\text{st}$ \cite[Lemma 2.2]{tanabe2019proximal}. We now \textcolor{black}{characterize} the fixed-point behavior of EPO-AL. 

\begin{theorem}[Fixed point analysis]  \label{theor}
    {\color{LH}Assume that} Algorithm~\ref{alg:PO_minmax} converges to a pair of fixed-points $w_\infty$ and $p_\infty$. Then \( w_{\infty} \) is both Pareto stationary and {\color{black}fair}.
\end{theorem}
\begin{proof}
We begin the proof by substituting $w_{\infty}$ for \( w_{i} \) and \( w_{i-1} \) as well as \( p_{\infty} \) for \( p_i \) and \( p_{i-1} \) in~\eqref{eq:primal}--\eqref{eq:dual}, which yields:
\begin{align}
    G(w_\infty)\big([p_\infty]_+ + \eta L_r \mathcal{J}(w_\infty)\big) = 0; \label{eq:primal_fixed}\\
    L_r \mathcal{J}(w_\infty) = 0.\label{eq:dual_fixed}
\end{align}
From (\ref{eq:dual_fixed}), \eqref{eq:primal_fixed} simplifies to:
\begin{align}
    G(w_\infty)[p_\infty]_+ = \sum_{k=1}^K [p_{k, \infty}]_+ \nabla J_k(w_{\infty}) = 0, \label{eq:primal_fixed_2}
\end{align}
which ensures that $w_{\infty}$ is Pareto stationary{\color{LH},} provided that \( [p_{\infty}]_{+} \) contains at least one non-zero entry, which we will investigate further below. In light of~\eqref{eq:helloooo}, the second condition (\ref{eq:dual_fixed}) implies that $w_\infty$ \textcolor{black}{yields} $r_1J_1(w_\infty) = \cdots = r_K J_K (w_\infty)$, and hence \( w_{\infty} \in \mathcal{F}_r \) {\color{black}satisfies the fairness condition}.

The only remaining step is to show that $p_{\infty}$ contains at least one strictly positive element. To this end, observe from~\eqref{eq:L_r} that \( r^{-1} \triangleq [ r_1^{-1}, \ldots, r_K^{-1}] \) is in the nullspace of \( L_r \):
\begin{align}
    L_r r^{-1} =&\: \text{diag}(r)\Big(I_{K\times K} -\frac{1}{K} \mathds{1}_K \mathds{1}_K^\top\Big)\text{diag}(r) r^{-1} \notag \\
    =&\: \text{diag}(r)\Big(I_{K\times K} -\frac{1}{K} \mathds{1}_K \mathds{1}_K^\top\Big)\mathds{1}_K = 0.
\end{align}
Hence, {\color{LH}by} taking the inner product of~\eqref{eq:dual} with \( r^{-1} \), we {\color{LH}have}:
\begin{align}
    (r^{-1})^{\mathsf{T}} p_i =&\: (r^{-1})^{\mathsf{T}} p_{i-1} + \mu (r^{-1})^{\mathsf{T}} L_r \mathcal{J}({w_{i-1}}) = (r^{-1})^{\mathsf{T}} p_{i-1}.
\end{align}
{\color{LH}Upon i}terating all the way back to \( p_0 \), we find {\color{LH}that}:
\begin{align}
    \sum_{k=1}^K \frac{p_{k, i}}{r_k} = \sum_{k=1}^{K} \frac{p_{k, 0}}{r_k} = \sum_{k=1}^{K} \frac{1}{K r_k} \triangleq \epsilon > 0.
\end{align}
For \( \sum_{k=1}^K {p_{k, i}}/{r_k} \) to be greater than \( \epsilon \), there must exist at least one \( k' \) such that \( {p_{k', i}}/{r_{k'}} \ge \epsilon \) and hence \( p_{k', i} \ge \epsilon r_{k'} \ge \epsilon \min_k r_k \). We conclude that for all \( i \) {\color{LH}we have}:
\begin{align}
    \max_k p_{k, i} \ge \left(\min_k r_k\right) \left(\sum_{k=1}^{K} \frac{1}{K r_k} \right).
\end{align}
{\color{LH}Upon a}ssuming that \( p_i \) approaches the fixed-point \( p_{\infty} \) and taking limits yields the desired result.
\end{proof}

\begin{corollary}[Convex objectives] \label{corr:fixed_point}
Assume that the objectives $J_k(w)$ are convex for all $k=1,\ldots,K$. Then, $w_\infty \in \mathcal{E}_r = \mathcal{P} \cap \mathcal{F}_r$ is exact Pareto optimal and solves the min-max problem~\eqref{eq:obj}.
\end{corollary}
\begin{proof}
    The result is immediate after recognizing that Pareto stationarity implies weak Pareto optimality \cite[Lemma 2.2]{tanabe2019proximal} for convex objectives. Hence, \( w^{\infty} \in \mathcal{P} \). Theorem~\ref{theor} already established that \( w^{\infty} \in \mathcal{F}_r\). Under these conditions, Proposition~\ref{prop:epo_minmax} ensures that \( w^{\infty} \in \mathcal{E}_r \) and also solves~\eqref{eq:obj}.
\end{proof}

\section{Empirical Evaluation}\label{sec:empirical}
We empirically evaluate our algorithm using {\color{LH}a pair of} synthetic experiments: when the objectives $\{J_k(w)\}_{k=1}^K$ are (\emph{i}) all convex and are (\emph{ii}) all non-convex\footnote{Code is available at \url{https://github.com/sangwoo-p/EPO_AL}}. Specifically, for the convex scenario, we consider (\emph{i}) $J_k(w)=\sqrt{1+||w-w_k||^2}-1$; for the non-convex scenario we take  (\emph{ii}) $J_k(w) = 1 - e^{-||w-w_k||^2}$ adapted from \cite{lin2019pareto} to deal with  more than two objectives.  Specifically, the $K$ \emph{anchor} points $\{w_k\}_{k=1}^K$ are chosen uniformly at random on the unit $(d-1)$-surface, and we also choose the preference vector $r$ by sampling uniformly at random in the {\color{black}{interior of the}} probability simplex {\color{black}$\Delta^{K}_{+}$} for which we define as {\color{black}$\Delta^{K}_{+}=\{ y \in \Delta^K: y_k > 1/3K \ \forall k \}$. We impose such strict positivity to avoid extreme cases where some objectives are essentially ignored.} We choose the initial model $w_0$ by randomly sampling from {\color{LH}the} unit $(d-1)$-sphere. We account for these randomnesses by running {\color{black}$30$} independent experiments, unless specified otherwise.

We compare the proposed algorithm {\color{LH}to} (\emph{i}) the subgradient {\color{black}algorithm} (\ref{eq:subgrad}) where the active index $k^\text{active}$ is chosen by breaking any tie at random; (\emph{ii}) \textcolor{black}{the} smooth-max approach \cite{lin2024smooth, gokcesu2019accelerating, hou2022joint} that updates $w \leftarrow w - \mu \nabla \mathrm{LSE}_\tau{\color{LH}[}r_1 J_1(w), \ldots, r_K J_K(w){\color{LH}],}$ where $\mathrm{LSE}_\tau[v_1,\ldots,v_K]$ is the smooth-max function defined as $\mathrm{LSE}_\tau[v_1,\ldots v_K] = \log \sum_{k=1}^K e^{v_k/\tau}$; and EPO Search \cite{mahapatra2020multi}{\color{LH},} which has the form of $w \leftarrow w - \mu G(w) \beta$ where $\beta \in \Delta^K$ is chosen by solving a $K$-dimensional linear program \cite{mahapatra2020multi} at every iteration.

 \begin{figure*}[t]
\begin{center}\includegraphics[scale=0.083]{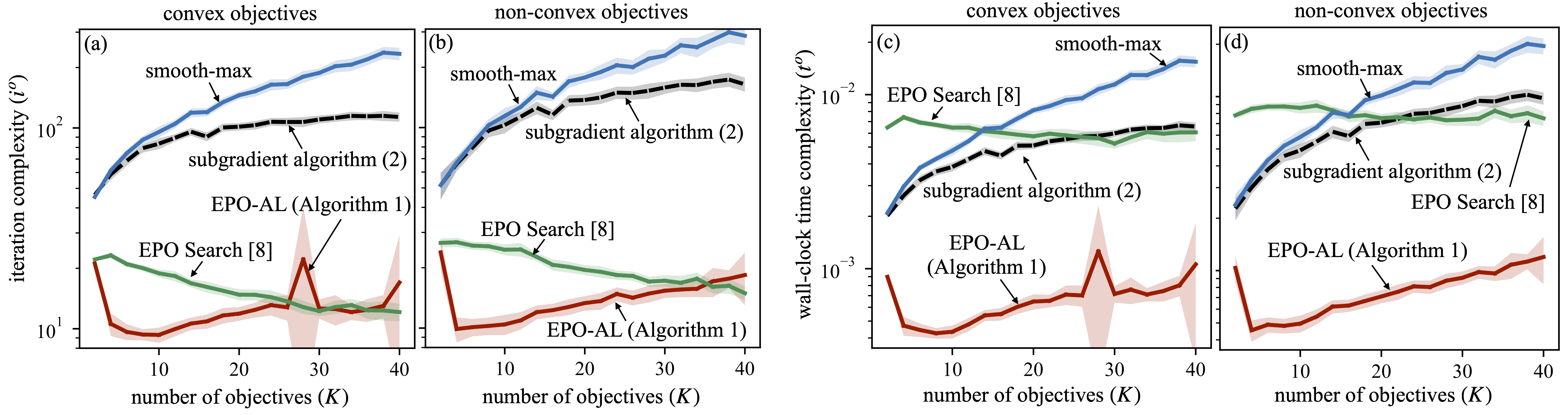}
    \end{center}
    \caption{Iteration complexity $i^o$ (a,b)  and wall-clock time complexity $t^o$ (c,d) as a function of number of objectives $K$. The results are averaged over {\color{black} $30$} independent experiments {\color{black}after removing the minimum and  maximum,} where each experiment assumes different preference vector $r$ and different initial model $w_0$. Shaded area corresponds to $99 \%$ confidence interval.}
    \label{fig:time_complexity}
    \vspace{-0.4cm}
 \end{figure*}


\subsection{Visualization of the optimization trajectory} 
\label{subsec:vis}
We first visualize the optimization trajectory of the {\color{LH}schemes} considered when all the objectives are non-convex \textcolor{black}{in Fig.~\ref{fig:vis}}. We omit the smooth-max approach as it fails to converge to the optimal point \cite{epasto2020optimal, gokcesu2019accelerating} for large enough $\tau$ that gives {\color{LH}us a} distinct optimization trajectory {\color{LH}along} with {\color{LH}the} subgradient {\color{black}algorithm} (\ref{eq:subgrad}). We set $d=3$ and $K=2$. {\color{LH}An} interesting observation here is that existing approaches prioritize converging to the fairness constraint before searching for the Pareto front, while the proposed EPO-AL algorithm rapidly converges to the Pareto front and then sweeps it for a solution that also {\color{black}satisfies the fairness condition}. We set $\mu=0.1$ for all the schemes {\color{LH}along} with $\eta=10$ for EPO-AL, $r=[0.2, 0.8]^{\top}$, and choose the two anchors following \cite{lin2019pareto}.


\subsection{Iteration/{\color{black}time} complexity}
We now consider the iteration complexity of the four {\color{LH}algorithms} considered  by measuring the minimum number of iterations {\color{LH}required} to achieve a specified target accuracy. In order to fairly compare different algorithms, we set the step size $\mu$ for each algorithm separately by searching over $\mathcal{G}_{\mu} = [10^{-3}, 10^{-1}]$ in a log-scaled grid of size $10$. {\color{LH}As for the} EPO-AL and smooth-max algorithm, we also set the penalty parameter $\eta$ and temperature parameter $\tau$ by searching over $\mathcal{G}_\eta = [10^{-1}, 10^2]$ and $\mathcal{G}_{\tau} = [10^{-2}, 10]$ respectively, both in a log-scaled grid of size $10$. We set the maximum number of iterations as $1000$ and set the dimension of the model as $d=100$, i.e., $w\in\mathbb{R}^{100}$.

Specifically, we define the target performance $J^\star$ as the minimum value attained by the subgradient  {\color{black}algorithm} (\ref{eq:subgrad}) throughout all of the possible step size choices. We then evaluate the iteration complexity for a fixed choice of $\mu$ (for all the algorithms) as well as of $\eta$ (for EPO-AL) and of $\tau$ (for smooth-max) by $i^o(\mu,\eta, \tau) = \min\{i:  | \max_k r_k J_k(w_i) - J^\star |\leq \epsilon\}$ with {\color{LH}the} tolerance level set {\color{LH}to} $\epsilon=0.01$. We then finally evaluate the iteration complexity $i^o$ for each scheme by choosing the minimum $i^o(\mu,\eta,\tau)$ among the possible choices of $\mu$, $\eta$, and $\tau$, i.e., $i^o = \min_{\mu \in \mathcal{G}_\mu} i^o(\mu,\eta,\tau)$ for {\color{LH}the} subgradient {\color{black} algorithm} (\ref{eq:subgrad}) and EPO Search; $i^o = \min_{(\mu,\tau)\in \mathcal{G}_\mu \times \mathcal{G}_\tau} i^o(\mu,\eta,\tau)$ for smooth-max; and $i^o = \min_{(\mu,\eta)\in \mathcal{G}_\mu \times \mathcal{G}_\eta} i^o(\mu,\eta,\tau)$ for EPO-AL.

Fig.~\ref{fig:time_complexity} (left) shows the iteration complexity as a function of the number of objectives $K$ for both convex and non-convex functions $J_k(w)$. It is observed that both {\color{LH}the} classical EPO Search \cite{mahapatra2020multi} and the proposed EPO-AL algorithm scale well with the number of objectives{\color{LH},} unlike the subgradient {\color{black}algorithm (\ref{eq:subgrad})} that scales poorly {\color{LH}upon} increasing {\color{LH}the} number of objectives.  {\color{black} Since the iteration count does not capture the computational complexity associated with each iteration (see Fig.~\ref{fig:vis}), we next investigate the minimum \emph{total} complexity required to reach the target performance $J^*$.

Fig.~\ref{fig:time_complexity} (right) shows the wall-clock time complexity $t^o$, defined as the actual total time required to process the number of iterations $i^o$. The wall-clock time is evaluated on Apple M1 hardware. The fact that EPO Search \cite{mahapatra2020multi} involves the solution of a linear program at every iteration results in higher per-iteration complexity and hence higher total runtime compared to the proposed EPO-AL strategy, which only involves a single timescale.
}

\section{Conclusion}
A new algorithm {\color{LH}was proposed} for min-max optimization via exact Pareto optimization. To derive the strategy we made use of primal-dual consensus techniques via the augmented Lagrangian, resulting in an algorithm which scales better with the number of objectives than a subgradient{\color{LH}-}based approach, while maintaining a lower per-iteration complexity than other smoothing{\color{LH}-}based algorithms. {\color{black}Experimental results showed that the proposed algorithm achieves the target performance by imposing lower total complexity as compared to the other benchmarks, demonstrating its scalability with the  number of objectives.}
\vspace{-0.1cm}
\bibliographystyle{ieeetr} 
\bibliography{ref.bib}

\end{document}